\documentclass{amsart}
\usepackage{bbm}
\usepackage{mathrsfs}
\usepackage{cases}
\usepackage{latexsym}
\usepackage[arrow,matrix]{xy}
\usepackage{stmaryrd}
\usepackage{amsfonts}
\usepackage{amsmath,amssymb,amscd,bbm,amsthm,mathrsfs,dsfont}
\usepackage{fancyhdr}
\usepackage{amsxtra,ifthen}
\usepackage{verbatim}

\numberwithin{equation}{section}

\theoremstyle{plain}
\newtheorem{prop}{Proposition}[section]
\newtheorem{thm}[prop]{Theorem}
\newtheorem{cor}[prop]{Corollary}
\newtheorem{lem}[prop]{Lemma}
\newtheorem{dfn}[prop]{Definition}

\theoremstyle{definition}
\newtheorem{example}[prop]{Example}
\newtheorem{remark}[prop]{Remark}
\newtheorem{remarks}[prop]{Remarks}

\newcommand{\lam}{\lambda}
\newcommand{\ft}{\mathfrak t}
\newcommand{\fs}{\mathfrak s}

\DeclareMathOperator{\rad}{rad} \DeclareMathOperator{\rank}{rank}
 \DeclareMathOperator{\res}{res}

\begin{document}

\title[Radicals of weight one blocks of Ariki-Koike algebras]
{Radicals of weight one blocks of Ariki-Koike algebras}

\thanks {Corresponding Author: Yanbo Li}
\thanks {Li is supported by the Natural Science Foundation of Hebei
Province, China (A2017501003) and NSFC 11871107.}

\author{Yanbo Li}

\address{Li: School of Mathematics and Statistics, Northeastern
University at Qinhuangdao, Qinhuangdao, 066004, P.R. China}

\email{liyanbo707@163.com}

\author{Dashu Xu}

\address{Xu: School of Mathematics and Statistics, Northeastern
University at Qinhuangdao, Qinhuangdao, 066004, P.R. China}

\email{1264942498@qq.com}

\begin{abstract}
Let $K$ be a field and $q\in K$, $q\neq 0, 1$. Let $\mathcal
{H}_n(q, Q)$ be an Ariki-Koike algebra, where the cyclotomic
parameter $Q=(Q_1, Q_2, \cdots, Q_r)\in K^r$ with $r\geq 2$,
$Q_i=q^{a_i}$, $a_i\in \mathbb{Z}$. For a weight one block $B$ of
$\mathcal {H}_n(q, Q)$, we prove in this paper that $\rad B=I$,
where $I$ is the nilpotent ideal constructed for a symmetric
cellular algebra in [Radicals of symmetric cellular algebras,
Colloq. Math. {\bf 133} (2013) 67-83]. We also give some applications of this result.
\end{abstract}



\maketitle

\section{Introduction}
Ariki-Koike algebras are cyclotomic Hecke algebras of type $G(r, 1,
n)$. They were introduced by Ariki and Koike in \cite{AK}, and
independently by Brou$\rm\acute{e}$ and Malle in \cite{BM}, which
include usual Hecke algebras of type $A_{n-1}$ and $B_n$ as special cases. The
cyclotomic Hecke algebras are central to the conjectures of
Brou$\rm\acute{e}$, Malle and Michel \cite{B}. However, these algebras are less well-understood
than Hecke algebras of type A.

\smallskip

It is well-known that the Specht modules of Hecke algebra of type A are indexed by partitions.
James introduced the so-called ``weight" for a partition, which turns out to be a block invariant.
Note that the weight of a block is a measure of the complexity of the representation theory of that
block. For example, a block has weight zero if and only if it is simple.
We refer the reader to \cite{F2, R, S, T} for more results about blocks of small weight.
In \cite{F1} Fayers generalized the notion of weight to multi-partitions and studied weight one blocks
 of Ariki-Koike algebras. The result about these blocks is similar to that of
Hecke algebras of type A, except the number of partitions in a block. A natural thought is to
generalize the known results about blocks of a Hecke algebra of type A to that of an Ariki-Koike algebra.
The purpose of this paper is to do something along this way.

In \cite{GL}, Graham and Lehrer introduced cellular algebras
and proved that the Ariki-Koike algebras are cellular.
In \cite{L}, Li studied the radical of a symmetric cellular algebra
by constructing a nilpotent ideal $I$. Along this way, Li \cite{L2}
studied the radical of the group algebra of a symmetric group by using Murphy basis and
proved that for $A=\mathbb{Z}_pS_n$ with $p$ a prime and $S_n$ a
symmetric group, if $n<2p$, then $\rad A=I$. In fact, the condition
$n<2p$ implies that the weight of all blocks of $A$ is less than
two. Let $\mathcal{H}_q(S_n)$ be a Hecke algebra over $S_n$.
Then each block of $\mathcal{H}_q(S_n)$
is a symmetric cellular algebra. A natural question is:
Let $B$ be a weight one block of $\mathcal{H}_q(S_n)$.
Is $\rad B$ equal to $I_B$, where $I_B$ is the
nilpotent ideal of $B$ constructed by certain cellular basis of $B$?
Furthermore,  is the similar result true for a weight one block of a symmetric Ariki-Koike
algebras?  Note that Malle and Mathas \cite{MM} proved that the Ariki-Koike algebras over any
ring containing inverses of the parameters are symmetric.

Our aim becomes possible depending on the work of Hu and Mathas \cite{HM}.
They gave a graded cellular basis (HM basis) for the symmetric Ariki-Koike
algebra $\mathcal {H}_n(q, Q)$. The particularly useful property of
HM basis to our goal is that it is compatible with the block
decomposition and consequently, each block of $\mathcal
{H}_n(q, Q)$ is a graded symmetric cellular algebra. The main result of
this paper is that for a weight one block $B$ of $\mathcal {H}_n(q,
Q)$, we have $\rad B=I_B$, where $I_B$ is the nilpotent ideal constructed
in \cite{L} by using HM basis of $B$.

\smallskip

Note that Martin's conjecture \cite{Ma}
claims that all projective (indecomposable) modules of a weight $w$ block
in $\mathbb{Z}_pS_n$ have a common radical length $2w+1$ if $w<p$.
This conjecture has been verified for $w\leq 3$.  For more details, see \cite{T}
and the references therein. It is helpful to point out that our result on a weight
one block $B$ in an Ariki-Koike algebra has a  direct corollary, that is,
$(\rad B)^3=0$. This implies that all projective
(indecomposable) modules of $B$ have a common radical length which is at most 3 ($=2w+1$).
However, we will see, the radical length of projective modules need not be $3$,
that is, according to our result, Martin's conjecture of cyclotomic version does not hold,
even for weight one blocks.

\medskip

\section{Symmetric cellular algebras and Ariki-Koike algebras}

Let $K$ be a field. Denote by $K^\times$ the nonzero elements of $K$.
Recall that a finite dimensional $K$-algebra $A$ is called symmetric
if there is a non-degenerate associative symmetric bilinear form $f$
on $A$. Define a $K$-linear map $\tau: A\rightarrow K$ by
$\tau(a)=f(a,1)$. We call $\tau$ a symmetrizing trace.

\subsection{Symmetric cellular algebras}
We refer the reader to \cite{GL} for the definitions of cellular
algebras and cell modules. Let $A$ be a cellular $K$-algebra with
cell datum $(\Lambda, M, \ast, C)$. Then $A$ has a cellular basis
$\{C^\lam_{S,T}\mid \lam\in\Lambda, S,T \in M(\lam)\}$. It is easy to check that
$$C_{S,T}^\lam C_{U,V}^\lam \equiv\Phi(T,U)C_{S,V}^\lam\,\,\,\, (\rm mod\,\,\, A(<\lam)),$$
where $\Phi(T,U)\in K$ depends only on $T$ and $U$. Define
$$\Lambda_0=\{\lam\in\Lambda\mid \text{there exist}\,\, S,T\in M(\lam)
\,\, \text{such that}\,\, \Phi(S,T)\neq 0 \}.$$

\smallskip

For a cell module $W(\lam)$, define a bilinear form $\Phi_{\lam}:
W(\lam)\times W(\lam)\rightarrow K$ by
$\Phi_{\lam}(C_{S},C_{T})=\Phi(S,T).$ The radical of the bilinear
form is defined to be
$$\rad\lam:= \{x\in W(\lam)\mid\Phi_{\lam}(x,y)=0\,\,\,\text{for all} \,\,\,y\in W(\lam)\}.$$
Define $L(\lam)$ to be $W(\lam)/\rad\lam$. If $\lam\in\Lambda_0$,
then $\rad\lam=\rad W(\lam)$. Furthermore, fix an order on $M(\lam)$
and define the Gram matrix $G(\lam)$ to be $(\Phi(T,U))_{T, U\in
M(\lam)}$. Then $\dim L(\lam)=\rank G(\lam)$. In particular,
$\{L(\lam)\mid \lam\in\Lambda_0\}$ is a complete set of simple
$A$-modules up to isomorphism. The composition multiplicity of
$L(\mu)$ in $W(\lam)$ will be denoted by $[W(\lam) : L(\mu)]$.

Now let $A$ be a finite dimensional symmetric cellular $K$-algebra
and $\tau$ a symmetrizing trace. Denote the dual basis by
$D=\{D_{S,T}^\lam \mid \lam\in\Lambda,  S,T\in M(\lam)\}$, which
satisfies
$\tau(C_{S,T}^{\lam}D_{U,V}^{\mu})=\delta_{\lam\mu}\delta_{SV}\delta_{TU}.$
Then Li proved in \cite{L} that for arbitrary elements $S,T,U,V\in
M(\lam)$,
$$D_{S,T}^\lam D_{U,V}^\lam \equiv \Psi(T,U)D_{S,V}^\lam\,\,\,\,
(\rm mod\,\,\, A(>\lam)),$$ where $\Psi(T,U)\in K$ depends only on
$T$ and $U$. Consequently, we also have Gram matrices $G'(\lam)$
defined by the dual basis. Furthermore, for any $\lam\in\Lambda$,
define $$k_{\lam}=\sum\limits_{X\in M(\lam)}\Phi(X,V)\Psi(X,V),$$
where $V\in M(\lam)$. Then Li proved in \cite{L} that
$G(\lam)G'(\lam)=k_{\lam}E,$ where $E$ is the identity matrix.
Moreover, for arbitrary $\lam\in \Lambda$ and $S\in M(\lam)$, we
have $(C_{S,S}^{\lam}D_{S,S}^{\lam})^2=k_\lam
C_{S,S}^{\lam}D_{S,S}^{\lam}.$ If $\tau(a)=\tau(a^\ast)$ for all
$a\in A$, where $\ast$ is the anti-automorphism of $A$, then the dual basis is cellular too. The corresponding
cell modules will be denoted by $W_D(\lam)$. We can define a bilinear form $\Psi_{\lam}:
W_D(\lam)\times W_D(\lam)\rightarrow K$ by
$\Psi_{\lam}(D_{S},D_{T})=\Psi(S,T).$

\smallskip

Moreover, Li defined some subsets of $\Lambda$ in \cite{L} as follows.

$\Lambda_{1}=\{\lam\in\Lambda\mid
\rad\lam=0\},$\qquad\qquad\qquad\qquad
$\Lambda_{2}=\Lambda_{0}-\Lambda_{1},$

$\Lambda_{3}=\Lambda-\Lambda_{0},$\qquad\qquad\qquad\qquad\qquad\qquad\,\,\,\,\,\,
$\Lambda_{4}=\{\lam\in\Lambda_{1}\mid k_{\lam}=0\}$.

Then the following lemma is clear and we omit the proof here.
\begin{lem}\label{2.1}
Let $A$ be a finite dimensional symmetric cellular algebra with a
cellular basis $\{C_{S,T}^\lam\mid\lam\in\Lambda, S, T\in M(\lam)
\}$. Let $B$ be a block of $A$, which is also a symmetric cellular
algebra with a basis $\{C_{U, V}^{\mu}\mid \mu\in\Lambda^B, U, V\in
M(\mu)\}$, where $\Lambda^B$ is some subset of $\Lambda$. Then
$\Lambda_i^B=\Lambda_i\bigcap\Lambda^B$ for $i=0, 1, 2, 3, 4$.
\end{lem}

\subsection{Ariki-Koike algebras}

Let $1\neq q\in K^{\times}$. Define the quantum characteristic of $q$
to be the positive integer $e$ which is minimal such that
$1+q+\cdots+q^{e-1}=0$. If no such $e$ exists, we set $e=0$. Suppose
that $Q=(Q_1, Q_2, \cdots, Q_r)\in K^r$ with $r\geq 2$,
$Q_i=q^{a_i}$, $a_i\in \mathbb{Z}$. The Ariki-Koike algebra
$\mathcal {H}_n(q, Q)$ with parameters $q$ and $Q$ is the unital
associative $K$-algebra with generators $T_0$, $T_1$, $\cdots$,
$T_{n-1}$ subject to the following relations:
\begin{enumerate}
\item[(H1)]\, $(T_0-Q_1)(T_0-Q_2)\cdots(T_0-Q_r)=0$;
\item[(H2)]\, $T_0T_1T_0T_1=T_1T_0T_1T_0$;
\item[(H3)]\, $(T_i+1)(T_i-q)=0\quad\quad\,\,\text{for}\,\, 1\le i \le n-1$;
\item[(H4)]\, $T_iT_{i+1}T_i=T_{i+1}T_iT_{i+1}\quad\,\text{for}\,\,1\le i \le n-2$;
\item[(H5)]\, $T_iT_j=T_jT_i\quad\quad\quad\quad\quad\,\,\,\,\,\text{for}\,\, 0\le i<j-1\le n-2$.
\end{enumerate}

In order to describe the cellular structure of $\mathcal {H}_n(q,
Q)$, let us first recall some combinatorics. Let $n$ be a positive
integer. A partition $\lam$ of $n$ is a non-increasing sequence of
non-negative integers $\lam=(\lam_1,\cdots,\lam_s)$ such that
$\sum_{i=1}^{s}\lam_i=n$ and we write $|\lam|=n$. The diagram of a
partition $\lam$ is the set of nodes $[\lam]=\{(i,j)\mid 1\leq i,
1\leq j\leq\lam_{i}\}$. An $r$-partition of $n$ is an $r$-tuple
$\lam=(\lam^{(1)}, \cdots, \lam^{(r)})$ of partitions such that
$\sum_{i=1}^r|\lam^{(i)}|=n$. The partitions $\lam^{(1)}, \cdots,
\lam^{(r)}$ are the components of $\lam$. Denote the set of
$r$-partitions of $n$ by $\mathscr{P}_{r,n}$. Then for $\lam,
\mu\in \mathscr{P}_{r,n}$, we say $\lam\unrhd \mu$ (or $\mu\unlhd\lam$) if
$$\sum_{t=1}^{s-1}|\lam^{(t)}|+\sum_{i=1}^j\lam_i^{(s)}
\geq\sum_{t=1}^{s-1}|\mu^{(t)}|+\sum_{i=1}^j\mu_i^{(s)}$$ for all
$1\leq s\leq r$ and all $j\geq 1$. Write $\lam\rhd\mu$ (or $\mu\lhd\lam$)  if
$\lam\unrhd\mu$ and $\lam\neq\mu$. The Young diagram of an
$r$-partition $\lam$ is the set of nodes $$[\lam]=\{(i, j, k)|1\leq
i, 1\leq j\leq\lam^{(k)}_i, 1\leq k\leq r\}.$$ A $\lam$-tableau is a bijective
map $\ft: [\lam]\rightarrow \{1, 2, \cdots, n\}$. A $\lam$-tableau $\ft$ is
called standard if the entries increase along each row and down each
column in each component. We often denote the number of standard $\lam$-tableaux by $n_{\lam}$.

The first cellular basis of $\mathcal {H}_n(q, Q)$ was given by
Graham and Lehrer  in \cite{GL} using Kazhdan-Lusztig basis of
$\mathcal {H}(S_n)$, the Hecke algebra of $S_n$. In \cite{DJM},
Dipper, James and Mathas constructed another cellular basis (DJM
basis), which is similar to the basis of $\mathcal {H}(S_n)$
introduced by Murphy \cite{Mr}. Then one has cell (Specht) modules
$W(\lam)$, where $\lam$ are $r$-partitions. Moreover, Ariki
\cite{A} proved that if $\lam$ is a Kleshchev $r$-partition, then
$L(\lam)$, the top of $W(\lam)$ is simple, and  $\{L(\lam)\mid
\lam\,\, \rm Kleshchev\}$ provide a complete set of simple $\mathcal
{H}_n(q, Q)$-modules up to isomorphism. It is helpful to point out that
the set of Kleshchev $r$-partitions is the set $\Lambda_0$ we discussed previously.

Note that the cellularity of $\mathcal {H}$ ensures that every cell
(Specht) module lies in one block and we abuse notation to say that
an $r$-partition $\lam$ lies in a block $B$ if $W(\lam)$ lies in
$B$.

In \cite{HM}, Hu and Mathas introduced graded cellular algebras and
proved the following results.
\begin{lem}\cite[Theorem 5.8, Corollary 5.12, Corollary
6.18]{HM}\label{2.2} Let $\mathcal {H}_n(q, Q)$ be a cyclotomic
Hecke algebra of type $G(r, 1, n)$. Then
\begin{enumerate}
\item[(1)] The algebra $\mathcal {H}_n(q, Q)$ is a graded cellular algebra
with poset $(\mathscr{P}_{r,n}, \unrhd)$ and graded cellular
basis $\{\psi_{\fs,\ft}^\lam\mid\lam\in\mathscr{P}_{r,n}, \fs,
\ft\in {\rm Std}(\lam)\}$  {\rm(HM basis)}.
\item[(2)] Let $\mathcal
{H}_\beta$ be a block of $\mathcal {H}_n(q, Q)$. Then there exists
$\mathscr{P}_{r, n}^{\mathcal{H}_\beta}\subseteq\mathscr{P}_{r, n}$ such
that $\{\psi_{\fs,\ft}^\lam\mid\lam\in\mathscr{P}_{r, n}^{\mathcal{H}_\beta},
\fs, \ft\in {\rm Std}(\lam)\}$ is a graded cellular basis of
$\mathcal {H}_\beta$. In particular, $\mathcal {H}_\beta$ is a
graded symmetric cellular algebra with homogeneous trace form
$\tau_\beta$ {\rm(\cite[Definition 6.15]{HM})}of degree $-2{\rm def}\beta$.
\end{enumerate}
\end{lem}

\begin{remarks}\label{2.3}
\begin{enumerate}
\item[(1)]\, It is easy to check that $\tau_\beta(a)=\tau_\beta(a^\ast)$ for all $a\in\mathcal
{H}_\beta$ and thus the dual basis of HM basis of $\mathcal{H}_\beta$ is cellular.
\item[(2)]\, For $\lam\in \mathscr{P}_{r,n}$, denote the graded cell (Specht)
module by $S^\lambda$. Hu and Mathas \cite{HM} pointed out that the
ungraded module $\underline{S}^\lambda$ coincides with the cell
module determined by the DJM basis.
\end{enumerate}
\end{remarks}

\medskip

\section{Radicals of weight one blocks}

We first recall the definition of the weight of
an $r$-partition given by Fayers in \cite{F1}. Let $\lam$
be an $r$-partition and let $Q$ be the cyclotomic
parameters of the algebra $\mathcal{H}_n(q;Q)$
as defined in Section 2.2. For $(i, j, k)\in[\lam]$, define the residue of the node
$(i, j, k)$ to be $\res\,(i, j, k)=q^{j-i}Q_k$ and define
$c_f(\lam)$ to be the number of nodes in $[\lam]$ of residue $f$.

\begin{dfn}\cite[(2.1)]{F1}\label{3.1}
Let $\lam$ be an $r$-partition of $n$. The weight of $\lam$ is
defined to be the integer
$$w(\lam)=(\sum_{i=1}^rc_{Q_i}(\lam))-\frac{1}{2}\sum_{f\in K^{\ast}}(c_f(\lam)-c_{qf}(\lam))^2.$$
\end{dfn}

\begin{remark}\label{3.2}
Note that two $r$-partitions $\lam$ and $\mu$ lie in the same block
if and only if $c_f(\lam)=c_f(\mu)$ for all $f\in K$ \cite{GL, LM}.
Then Definition \ref{3.1} implies that if $\lam$ and $\mu$ lie in
the same block, then $w(\lam)=w(\mu)$. So the weight of a block $B$
is defined to be the weight of any $r$-partition in $B$. Moreover,
if $\lam$ lies in a block $\mathcal {H}_\beta$, then $w(\lam)={\rm
def}\beta$ (see \cite{HM} for details).
\end{remark}

In order to prove our main result, we need to deal with the conjugate
of an $r$-partition. Let $\lam$ be a partition
of $n$. The conjugate of $\lam$ is defined to be a partition
$\lam'=(\lam'_1,\lam'_2,\cdots)$, where $\lam'_j$ is equal to the
number of nodes in column $j$ of $[\lam]$ for $j=1,2,\cdots$. The
conjugate of an $r$-partition $\lam$ is defined to be
$\lam=(\lam^{(r)'}, \cdots, \lam^{(1)'})$.
It is helpful to point out that in general $w(\lam)\neq w(\lam')$,
even for weight one $r$-partitions.
Let us give an example here.

\begin{example}\label{3.3}
Let $n=22$, $r=4$ and $e=9$. Assume that $Q_1=q$, $Q_2=q$,
$Q_3=q^5$ and $Q_4=q^2$. Take a $4$-partition $\lam=((3, 3, 2), (2, 1), (1^6) (2, 2, 1))$.
Then by direct computations $w(\lam)=1$ and $w(\lam')=6$.
\end{example}

However, we have the following lemma for a bipartition.

\begin{lem}\label{3.4}
Let $\lam$ be a bipartition. Then $w(\lam)=w(\lam')$.
\end{lem}

\begin{proof}
Let $\lam=(\lam^{(1)}, \lam^{(2)})$ be a bipartition. Then
$\lam'=(\lam^{(2)'}, \lam^{(1)'})$. Let $Q_1=q^{a_1}$,
$Q_2=q^{a_2}$. Define a bijection $\Theta: [\lam]\longrightarrow
[\lam']$ by $\Theta(i, j, k)=(j, i, 3-k)$. It is easy to check that
$\res(i, j, k)\res(j, i, 3-k)=Q_1Q_2.$ This implies that for $s=1,
2$ if $(i, j, k)$ is a node of residue $Q_s$ in $\lam$, then
$\Theta(i, j, k)$ is a node of residue $Q_{(3-s)}$ in $\lam'$.
Consequently,
$$\sum_{s=1}^2c_{Q_s}(\lam)=\sum_{s=1}^2c_{Q_s}(\lam').$$

On the other hand, for arbitrary $0\neq f\in K$, using the bijection
$\Theta$ again, we have $$c_f(\lam)=c_{q^{a_1+a_2}f^{-1}}(\lam')$$
and $$c_{qf}(\lam)=c_{q^{a_1+a_2-1}f^{-1}}(\lam').$$ Define a map
$\eta$ on $K^\ast$ by $\eta(f)=q^{a_1+a_2-1}f^{-1}$. Clearly, $\eta$
is a bijection and
$$|c_f(\lam)-c_{qf}(\lam)|=|c_{\eta(f)}(\lam')-c_{q\eta(f)}(\lam')|.$$
This gives that
\begin{eqnarray*}
\sum_{f\in K^\ast}(c_f(\lam)-c_{qf}(\lam))^2&=&\sum_{\eta(f)\in
K^\ast} (c_{\eta(f)}(\lam')-c_{q\eta(f)}(\lam'))^2\\
&=&\sum_{f\in K^\ast}(c_f(\lam')-c_{qf}(\lam'))^2
\end{eqnarray*}
and then the lemma follows from Definition \ref{3.1}.
\end{proof}

Although for $r$-partitions ($r\geq 3$), $w(\lam)\neq w(\lam')$ in general,
the following fact given by Fayers in \cite{F} is useful.

\begin{lem}\label{3.5}
Let $\lam$ be an $r$-partition. If $w(\lam)=c$ with parameters $q, Q_1, Q_2, \dots, Q_r$,
then $w(\lam')=c$ with parameter $q^{-1}, Q_r, Q_{r-1}, \dots, Q_1$.
\end{lem}

We also need the the following lemma proved in \cite{LR} by Lin and Rui.

\begin{lem}\cite[Lemma 2.4]{LR}\label{3.6}
Let $\lam$ and $\mu$ be two $r$-partitions of $n$. Then
$\lam\unrhd\mu$ if and only if $\mu'\unrhd\lam'$.
\end{lem}

Before we study the radical of weight one blocks of
$\mathcal{H}_n(q, Q)$, let us recall a general result on radicals of
symmetric cellular algebras given by Li in \cite{L2}.

\begin{lem}\cite[Lemma 3.1]{L2}\label{3.7}
Let $A$ be a symmetric cellular algebra. Suppose that
\begin{enumerate}
\item[(1)]\,$\sum\limits_{\lam\in\Lambda_{3}}n_{\lam}^{2}=
\sum\limits_{\lam\in\Lambda_{4}}n_{\lam}^{2},$ where $n_\lam$
is the number of the elements in $M(\lam)$;
\item[(2)]\,$\sum\limits_{\lam\in\Lambda_{2}}(\dim_{K}\rad\lam)^{2}
=\sum\limits_{\lam\in\Lambda_{2}}(\dim_{K}L_{\lam})^{2}.$
\end{enumerate}
Then $\rad A=I$, where $I$ is the nilpotent ideal of $A$ defined in \cite{L}.
\end{lem}

From now on, we investigate the Ariki-Koike algebra
$\mathcal{H}_n(q, Q)$ using the graded cellular basis
$\psi_{\fs\ft}^\lambda$ (HM basis). According to Remark \ref{2.3}, the results
obtained in \cite{F1} on Specht modules still hold for
$\underline{S}^\lambda$. For convenience, we denote the cell modules
by $W(\lam)$.

Fayers \cite{F1} studied decomposition numbers of weight one blocks
of $\mathcal{H}_n(q, Q)$. His result is a useful tool for
characterizing the subset $\Lambda_2$.

\begin{lem}\cite[Theorm 4.12]{F1} \label{3.8}
Suppose that $B$ is a block of $\mathcal{H}_n(q, Q)$ of weight one.
Then the $r$-partitions in $B$ are
$\lam_1\vartriangleleft\cdots\vartriangleleft\lam_s$ for some
$s\leqslant e$. All these $r$-partitions are Kleshchev except
$\lam_s$, and the decomposition number
$$[W(\lam_i) : L(\lam_j)]=\begin{cases}
1,& \text{if}\,\, i=j\,\, \text{or}\,\, i=j+1;\\
0,  & \text{otherwise}.
\end{cases}
$$
\end{lem}

Now we can completely describe $\Lambda^B_j$
for $0\leq j\leq 4$ for a given block of $\mathcal{H}_n(q, Q)$ of weight one.

\begin{lem}\label{3.9}
Suppose $B$ is a block of $\mathcal{H}_n(q, Q)$ of weight 1 and the
$r$-partitions in $B$ are
$\lam_1\vartriangleleft\cdots\vartriangleleft\lam_s$. Then:
\begin{enumerate}
\item[(1)]\,$\Lambda_0^B=\{\lam_1, \cdots, \lam_{s-1}\}$;
\item[(2)]\,$\Lambda_1^B=\{\lam_1\}$;
\item[(3)]\,$\Lambda_2^B=\{\lam_2, \cdots, \lam_{s-1}\}$.
\item[(4)]\,$\Lambda_3^B=\{\lam_s\}$;
\item[(5)]\,$\Lambda_4^B=\{\lam_1\}$.
\end{enumerate}
Moreover, $W(\lam_s)$ is simple and $\rad(W(\lam_i))\cong
L(\lam_{i-1})$ for $2\leq i\leq s$.
\end{lem}

\begin{proof}
These results all follow immediately from Lemma
\ref{3.10} except for the fact that $k_{\lam_1}= 0$.
In fact, since $\lam_1$ is minimal, in view of \cite[Lemma 2.15]{M},
$\lam_1\in \Lambda_1^B$. Moreover, since $B$ is of weight one, it
follows from Lemma \ref{2.2} that $B$ is a graded symmetric cellular
algebra with homogeneous trace form of degree $-2$. We have from
\cite[Theorem 3.4]{LZD} that none of the cell modules are projective. Then one
deduce by \cite[Theorem 4.4]{LX} that $k_{\lam_i}= 0$ for $i=1, \cdots, s$. Thus
$\lam_1\in \Lambda_4^B$.
\end{proof}

Let $\mathcal {H}'_n(q, Q)$ denote the Ariki-Koike algebra
with parameters $q^{-1}, Q_r, \cdots, Q_1$. To distinguish from $\mathcal {H}_n(q, Q)$,
we denote the $r$-partitions by $\overline{\lam}$ for $\mathcal {H}'_n(q, Q)$.
It is necessary to note that as $r$-partitions, $\lam=\overline{\lam}$.
Then we have the following lemma, whose proof is obvious by Lemma \ref{3.5}, \ref{3.6} and \ref{3.9}.

\begin{lem}\label{3.10}
Suppose $B$ is a block of $\mathcal{H}_n(q, Q)$ of weight 1 and the
$r$-partitions in $B$ are
$\lam_1\vartriangleleft\cdots\vartriangleleft\lam_s$. Then there
is a weight 1 block $\overline{B}$ of $\mathcal {H}'_n(q, Q)$
and the $r$-partitions in $\overline{B}$ are $\overline{\lam'_s}
\vartriangleleft\cdots\vartriangleleft\overline{\lam'_1}$.
\begin{enumerate}
\item[(1)]\,$\Lambda_0^{\overline{B}}=\{\overline{\lam'_s}, \cdots, \overline{\lam'_{2}}\}$;
\item[(2)]\,$\Lambda_1^{\overline{B}}=\{\overline{\lam'_s}\}$;
\item[(3)]\,$\Lambda_2^{\overline{B}}=\{\overline{\lam'_{s-1}}, \cdots, \overline{\lam'_{2}}\}$.
\item[(4)]\,$\Lambda_3^{\overline{B}}=\{\overline{\lam'_1}\}$;
\item[(5)]\,$\Lambda_4^{\overline{B}}=\{\overline{\lam'_s}\}$.
\end{enumerate}
Moreover, $W(\overline{\lam'_1})$ is simple and $\rad(W(\overline{\lam'_i}))\cong
L(\overline{\lam'_{i+1}})$ for $2\leq i\leq s$.
\end{lem}

Define a new algebra ${\rm H}_n=\mathcal{H}_n(q, Q)\bigoplus\mathcal {H}'_n(q, Q)$.
Clearly, ${\rm H}_n$ is a cellular algebra. Putting the HM basis of $\mathcal{H}_n(q, Q)$
and $\mathcal {H}'_n(q, Q)$ together gives a cellular basis of ${\rm H}_n$.
Denote by $\mathcal {B}$ the subalgebra of ${\rm H}_n$, which is the direct
sum of all blocks of weight 1 in $\mathcal{H}_n(q, Q)$ and $\mathcal {H}'_n(q, Q)$.
Then according to Lemma \ref{2.2}, $\mathcal {B}$ is a graded cellular
algebra with a cellular basis being the union of the cellular basis of the blocks.
Denote the poset by $\Lambda^{\mathcal {B}}$ and define $\Lambda_i^{\mathcal {B}}$ naturally for $i=0,1,2,3,4$.
We check the conditions (1) and (2) of
Lemma \ref{3.7} for $\mathcal {B}$. Let us divide the process into
two lemmas.

\begin{lem}\label{3.11}
Keep notations as above. Then
$\sum\limits_{\lam\in\Lambda_{3}^\mathcal
{B}}n_{\lam}^{2}=\sum\limits_{\lam\in\Lambda_{4}^\mathcal
{B}}n_{\lam}^{2}$, where $n_{\lam}$ is the number of standard $\lam$-tableaux.
\end{lem}

\begin{proof}
Let $\lam$ be an arbitrary weight one $r$-partition in a weight one block $B$ in $\mathcal{H}_n(q, Q)$.
Then Lemma \ref{3.5} implies that $\overline{\lam'}$ is in a weight one block $\overline{B}$ in $\mathcal{H}'_n(q, Q)$.
By Lemma \ref{3.9}, the $r$-partitions in $B$ are
$\lam_1\vartriangleleft\cdots\vartriangleleft\lam_s$, and
 $\lam_1\in\Lambda_4^B$ and
$\lam_s\in\Lambda_3^B$. On the other hand, we have from Lemmas
\ref{3.6} and \ref{3.10} that the $r$-partitions in $\overline{B}$ are
$\overline{\lam_s'}\vartriangleleft\cdots\vartriangleleft\overline{\lam_1'}$. According to
Lemma \ref{3.10}, $\overline{\lam_s'}\in \Lambda_4^{\overline{B}}$ and $\overline{\lam_1'}\in
\Lambda_3^{\overline{B}}$. Then Lemma \ref{2.1} shows that
$\lam_1\in\Lambda_4^\mathcal {B}$, $\lam_s\in\Lambda_3^\mathcal
{B}$, $\overline{\lam_s'}\in \Lambda_4^\mathcal {B}$ and $\overline{\lam_1'}\in
\Lambda_3^\mathcal {B}$. Note that $n_{\lam}=n_{\lam'}$. Letting
$\lam$ run over all of the weight one $r$-partitions, the lemma
follows.
\end{proof}

The proof of the next lemma is similar to that of \cite[Theorem
3.5]{L2}. For the convenience of readers, we write it out here.

\begin{lem}\label{3.12}
$\sum\limits_{\lam\in\Lambda_{2}^\mathcal {B}}(\dim_{K}\rad\lam)^{2}
=\sum\limits_{\lam\in\Lambda_{2}^\mathcal {B}}(\dim_{K}L_{\lam})^{2}.$
\end{lem}

\begin{proof}
For an $r$-partition $\lam$ of weight one, we first fix some
notations. Denote $\dim_K\rad\lam$ by $\mathbbm{r}_\lam$ and denote
$\dim_KL_\lam$ by $\mathbbm{l}_\lam$. Let $B$ be a weight one block  in $\mathcal{H}_n(q, Q)$
and let $\lam_1\vartriangleleft\cdots\vartriangleleft\lam_s$ be all
of the $r$-partitions in $B$. Then Lemma \ref{3.8} and Lemma
\ref{3.9} give that $\mathbbm{r}_{\lam_i}=\mathbbm{l}_{\lam_{i-1}}$
for $3\leq i \leq s-1$. Similarly, in block $\overline{B}$,
$\mathbbm{r}_{\overline{\lam_{i}'}}=\mathbbm{l}_{\overline{\lam_{i+1}'}}$ for $2\leq i\leq
s-2$. Moreover, employing Lemma \ref{3.8}, Lemma \ref{3.9} and Lemma \ref{3.10}
yields $\mathbbm{r}_{\lam_2}=n_{\lam_1}$,
$\mathbbm{r}_{\overline{\lam_{s-1}'}}=n_{\overline{\lam_s'}}$,
$n_{\overline{\lam_1'}}=\mathbbm{l}_{\overline{\lam_2'}}$ and
$n_{\lam_s}=\mathbbm{l}_{\lam_{s-1}}$. Note that
$n_{\lam}=n_{\overline{\lam'}}$ and thus
$\mathbbm{r}_{\lam_2}=\mathbbm{l}_{\overline{\lam_2'}}$ and
$\mathbbm{r}_{\overline{\lam_{s-1}'}}=\mathbbm{l}_{\lam_{s-1}}$. Therefore,
$\sum\limits_{i=2}^{s-1}(\mathbbm{r}_{\lam_i}^2+\mathbbm{r}_{\overline{\lam_i'}}^2)
=\sum\limits_{i=2}^{s-1}(\mathbbm{l}_{\lam_i}^2+\mathbbm{l}_{\overline{\lam_i'}}^2)$.
Letting $\lam$ run over all the $r$-partitions of weight one, the
lemma follows.
\end{proof}

Combining Lemma \ref{3.11} with Lemma \ref{3.12}, we get that $\rad
\mathcal {B}=I_\mathcal{B}$, where $I_{\mathcal {B}}$ is the
nilpotent ideal constructed by using the graded cellular basis (HM basis). It
is well known that the radical of an algebra is equal to the direct
sum of the radicals of blocks of the algebra. Then we have proven
the main result of this paper.

\begin{thm}\label{3.13}
Let $B$ be a weight one block in $\mathcal{H}_n(q, Q)$. Then
$\rad B=I_B$.
\end{thm}

\begin{cor}\label{3.14}
Let $B$ be a weight one block in $\mathcal{H}_n(q, Q)$ and the
$r$-partitions in $B$ are $\lam_1\vartriangleleft\cdots\vartriangleleft\lam_s$.  We have
\begin{enumerate}
\item[(1)]\, $(\rad B)^3=0$.
\item[(2)]\,If $s>2$, then $(\rad B)^2 \neq 0$.
\end{enumerate}
\end{cor}

\begin{proof}
(1) Note that Li proved in \cite{L} that for a symmetric cellular algebra,
$I^3=0$, where $I$ is the nilpotent ideal constructed in \cite{L}.
Then we have from Theorem \ref{3.13} that $(\rad B)^3=0$.

(2) Recall that in \cite{L2}, Li proved the following result:
if there exists $\lam$ with $k_{\lam}=0$ such that
$\Phi_{\lam}\neq 0$ and $\Psi_{\lam}\neq 0$, then
$(\rad A)^2\neq 0$. So we only need to find such $\lam$.
In fact, we have pointed out in the proof of Lemma \ref{3.9} that $k_{\lam_i}=0$ for $i=1,\cdots, s$.
Furthermore, in view of Lemma \ref{3.9}, the number of $\lam$ with $\Phi_\lam\neq 0$ is $s-1$.
Since the dual basis of HM basis of $B$ is also cellular, the number of $\lam$ with $\Psi_\lam= 0$ is 1.
Note that $s>2$. Then there exists $\lam$ with $\Phi_\lam\neq 0$ and $\Psi_\lam\neq 0$.
\end{proof}

\begin{remarks}\label{3.15}
\begin{enumerate}
\item[(1)]\,If $s=2$, then $(\rad B)^2=0$. In fact, in this case
$\Phi_{\lam_2}=0$ and $\Psi_{\lam_1}=0$ by Lemma \ref{3.8}. According to
the definition of $I$, in order to prove $(\rad B)^2=0$, we only need to
check that the product of arbitrary two generators of $I$ is zero.
This can be done by direct computations. The key technology is the formulas proved in
\cite[Lemma 3.1]{L} and \cite[3.1]{LZ}. We omit the details.
\item[(2)]\,If $e=2$, then for a weight one block $B$,
we have from Lemma \ref{3.8} that $s\leq 2$. Note that a block
has weight 0 if and only if it contains exactly one $r$-partition.
This implies that $s=2$ because $B$ is a weight one block. This yields  $(\rad B)^2=0$ by (1).
\end{enumerate}
\end{remarks}

In order to give another corollary, let us recall some definitions.
Let $A$ be a finite dimensional symmetric cellular algebra, whose center
is denoted by $Z(A)$. In \cite{L1}, Li defined an ideal $L(A)$ of $Z(A)$, which
is generated by $\{e_{\lam}\mid\lam\in\Lambda\}$, where
$e_{\lam}=\sum_{S\in M(\lam)}C_{S,T}^{\lam}D_{T,S}^{\lam}$. Another ideal of $Z(A)$,
Reynold ideal $R(A)$, is defined to be the intersection of $Z(A)$ with
the socle of $A$. In general, we only know $H(A)\subseteq L(A)\cap R(A)$,
where $H(A)$ is the so-called Higman ideal. Note that $H(A)$ is an ideal of $Z(A)$ too.

\begin{cor}
$L(B)=R(B)$.
\end{cor}

\begin{proof}
It follows from Theorem \ref{3.13} and the fact proved by Li
in \cite{L2}: for a symmetric cellular algebra $A$, if $\rad A=I$, then $L(A)=R(A)$.
\end{proof}

\medskip

\noindent{\bf Acknowledgement}
We are sincerely grateful to the
anonymous referee for the careful reading and valuable comments,
especially for pointing out an error in a preliminary version of this paper.
Yanbo Li would like to express his
sincere thanks to Chern Institute of Mathematics of Nankai
University for the hospitality during his visit.


\begin{thebibliography}{}

\bibitem{AK} S. Ariki and K. Koike, {\em A Hecke algebra of
$(\mathbb{Z}/r\mathbb{Z})\wr\mathfrak{S}_n$ and construction of its
irreducible representations},  Adv. Math. {\bf 106} (1994), 216-243.

\bibitem{A} S. Ariki, {\em On the classification of simple modules for cyclotomic
Hecke algebras of type $G(r, 1, n)$ and Kleshchev multi-partitions},
Osaka J. Math. {\bf 38} (2001), 827-837.

\bibitem{B} Brou$\rm\acute{e}$, {\em Reflection groups, braid groups, Hecke algebras, finite
reductive groups}, in: Current Developments in Mathematics, 2000,
International Press, Boston, 2001, pp. 1-103.

\bibitem{BM} M. Brou$\rm\acute{e}$ and G. Malle, {\em Zyklotomische Hecke algebren},
Ast\'{e}risque {\bf 212} (1993), 119-189.

\bibitem{DJM} R. Dipper, G. James and A. Mathas, {\em Cyclotomic $q$-Schur
algebras,} Math. Z. {\bf 229} (1999), 385-416.

\bibitem{F1} M. Fayers, {\em Weights of multipartitions and representations of
Ariki-Koike algebras}, Adv. Math. \textbf{206} (2008), 112-144.

\bibitem{F} M. Fayers, {\em Weights of multipartitions and representations of
Ariki-Koike algebras II: canonical bases}, J. Algebra \textbf{319} (2008), 2963-78.

\bibitem{F2} M. Fayers, {\em Decomposition numbers for weight three blocks of symmetric groups and Iwohori-Hecke algebras},
Trans. Amer. Math. Soc. \textbf{360} (2008), 1341-1376.

\bibitem{GL} J. Graham and G. Lehrer, {\em Cellular algebras},
 Invent. Math. {\bf 123} (1996), 1-34.

\bibitem{HM} J. Hu and A. Mathas, {\em Graded cellular bases for the
cyclotomic Khovanov-Lauda-Rouquier algebras of type $A$}, Adv. Math.
{\bf 225} (2010), 598-642.

\bibitem{L1} Yanbo Li, {\em Centers of symmetric cellular
algebras}, Bull. Aust. Math. Soc. {\bf 82} (2010), 511-522.

\bibitem{L} Y. Li, {\em Radicals of symmetric cellular
algebras}, Colloq. Math. {\bf 133} (2013), 67-83.

\bibitem{L2} Y. Li, {\em On the radical of the group
algebra of a symmetric group}, J. Algebra Appl. {\bf 16} (2017),
1750175 (11 pages).

\bibitem{LX} Y. Li and Z. Xiao, {\em On cell modules of symmetric cellular
algebra}, Monatsh. Math.  {\bf 168} (2012),  49-64.

\bibitem{LZ} Y. Li and D. Zhao, {\em Projective cell modules of Frobenius cellular algebras},
Monatsh. Math.  {\bf 175} (2014), 283-291.

\bibitem{LZD} Y. Li and D. Zhao, {\em On graded symmetric cellular algebras},
J. Aust. Math. Soc. (preprint) doi:10.1017/S1446788719000223.

\bibitem{LR} Z. Lin and H. Rui, {\em Cyclotomic q-Schur algebras and Schur-Weyl
duality}, Comtemp. Math. {\bf 413} (2006), 133-155.

\bibitem{LM} S. Lyle and A. Mathas, {\em Blocks of cyclotomic Hecke
algebras}, Adv. Math. {\bf 216} (2007), 854-878.

\bibitem{M} A. Mathas, {\em Iwahori-Hecke algebras and Schur algebras of the
symmetric group}, University Lecture Series 15, Amer. Math. Soc. (1999)

\bibitem{Ma} S. Martin, {\em Projective indecomposable modules for symmetric groups}, I. Q. J. Math. Oxford {\bf 44} (1993), 87-99.

\bibitem{Mr} G. Murphy, {\em The representations of Hecke algebras of type $A_n$}, J.
Algebra {\bf 173} (1995), 97-121.

\bibitem{MM} G. Malle and A. Mathas, {\em Symmetric cyclotomic Hecke algebras},  J. Algebra {\bf 205} (1998),
275-293.

\bibitem{R} M. J. Richards, {\em Some decomposition numbers for Hecke algebras of general linear groups},
Math. Proc. Cambridge Philos. Soc. {\bf 119} (1996), 383-402.

\bibitem{S} J. C. Scopes, {\em Cartan matrices and Morita equivalence for blocks of the symmetric groups}, J.
Algebra {\bf 142} (1991), 441-455.

\bibitem{T} K. Tan, {\em Martin¡¯s conjecture holds for weight 3 blocks
of symmetric groups}, J. Algebra {\bf 320} (2008), 1115-1132.

\end{thebibliography}
\end{document}